\documentclass[12pt,reqno]{amsart}
\usepackage{amssymb}
\usepackage{amsmath}
 \usepackage[usenames,dvipsnames]{color}
 \usepackage{pgf,tikz}
 \usetikzlibrary{decorations.markings}
\usetikzlibrary{arrows,shapes,positioning,calc,shadows}
\tikzstyle arrowstyle=[scale=1]
\tikzstyle directed=[postaction={decorate,decoration={markings,
    mark=at position .5 with {\arrow[arrowstyle]{stealth}}}}]
\usepackage{fullpage}
\usepackage{amsthm}
\usepackage{mathtools}
\usepackage[most]{tcolorbox}
\usepackage[T1]{fontenc}
\usepackage{graphicx}
\usepackage[applemac]{inputenc}
\usepackage[mathcal]{euscript}
\usepackage{wrapfig}
\usepackage[colorlinks=true, citecolor=cyan,urlcolor=blue,linktocpage=true,
	pagebackref,hyperindex=true,hypertexnames=true,hyperfigures=true]{hyperref} 
	
\usepackage{genyoungtabtikz}
\YFrench
\Yboxdim{1cm}
\Yfillcolour{green!80}
\definecolor{azure}{rgb}{0.,0.5,1.0} 
\definecolor{green}{rgb}{0.,0.6,0.} 

\newcommand\ylw{\Yfillcolour{yellow}}

\Ylinethick{.02cm}
 
\def\rouge{\textcolor{red}}
\def\bleu{\textcolor{blue}}

\def\auteur#1{{\sc #1}}
\def\titreref#1{{\em #1}}
\def\vol#1{{\bf #1}}

\newtheorem{conjecture}{\bleu{Conjecture}}

\newtheorem{property}{\bleu{Property}}
\newtheorem{properties}{\bleu{Properties}}

\newtheorem{proposition}{\bleu{Proposition}}

\numberwithin{equation}{section}
\numberwithin{lemma}{section}
\numberwithin{proposition}{section}
\numberwithin{remark}{section}
\numberwithin{property}{section}
\numberwithin{properties}{section}
\numberwithin{cor}{section}

\newenvironment{formula}{\begin{equation}}{\end{equation}}

\newcommand{\define}[1]{\bleu{\bf{#1}}}

\newcommand{\A}{\mathcal{A}}
\newcommand{\area}{\mathrm{area}}
\newcommand{\BarE}{\overline{\mathcal{E}}}
\newcommand{\E}{\mathcal{E}}
\newcommand{\F}{\mathcal{F}}
\newcommand\GL{\mathrm{GL}}
\newcommand{\length}{\boldsymbol{\ell}}
\newcommand{\maxmult}{\boldsymbol{m}}
\renewcommand\S{\mathbb{S}}
\newcommand\SchurHat{{\widehat{s}}}

\title{Triangular Diagonal Harmonics Conjectures.
}
\author{F.~Bergeron}
\address{\href{http://bergeron.math.uqam.ca}{D\'epartement de Math\'ematiques, Lacim, UQAM.}}
  \email{\href{mailto:bergeron.francois@uqam.ca}{bergeron.francois@uqam.ca}}
  \date{\bleu{\bf \today}. This work was supported by NSERC}

\begin{document}


\begin{abstract} The purpose of this paper is mostly to present conjectures that extend, to the ``triangular partition'' context (partitions ``under any line'' in the terminology of \cite{AnyLine}), properties of Frobenius of multivariate diagonal harmonics modules.
\end{abstract}

\maketitle
 \parskip=0pt
{ \setcounter{tocdepth}{2}\parskip=0pt\footnotesize \tableofcontents}
\parskip=8pt  
\parindent=20pt

\section{\bleu{Introduction}}
Triangular partitions and their properties are described and discussed in~\cite{Bergeron_Mazin}. 
The triangular partitions of size at most $5$ (and respective Ferrers diagrams in French notation) are: 
\ylw
\Ylinethick{.01cm}
 \Yboxdim{.2cm}
\begin{align*}
&0;&&\emptyset\\
&1;&& \yng(1);\\
&11,2;&&\yng(1,1),\ \yng(2);\\ 
&111,21,3;&&\yng(1,1,1),\ \yng(2,1),\ \yng(3);\\ 
&1111,211,31,4;&&\yng(1,1,1,1),\ \yng(2,1,1),\ \yng(3,1),\ \yng(4);\\ 
&11111,2111,221,32,41, 5;&&\yng(1,1,1,1,1),\ \yng(2,1,1,1),\ \yng(2,2,1),\ \yng(3,2),\ \yng(4,1),\ \yng(5).\\ 
\end{align*}
Classical cases include the \define{staircases} shapes $\delta(n)=(n-1,n-2,\ \ldots\ ,2,1)$:
\Ylinethick{.01cm}
 \Yboxdim{.3cm}
 \Yvcentermath1
$$\delta(1)=\hbox{$\emptyset$},\quad\delta(2)=\hbox{ \yng(1)},\quad \delta(3)=\hbox{ \yng(2,1)},\quad \delta(4)=\hbox{\yng(3,2,1)},\quad \delta(5)=\hbox{\yng(4,3,2,1)}, \quad \cdots $$
as well as the general staircases $((n-1)k,(n-2)k,\ \ldots\ , 2k,k)$, and their conjugate.
As a matter of fact, the set of triangular partitions is closed under conjugation. We respectively denote by $\length(\tau)$ and $\maxmult(\tau)$, the length of $\tau$ (number of non-zero parts), and the maximal multiplicity of a part in $\tau$. Thus, for the length $6$ partition $\tau=222211$, the value of $\maxmult(\tau)$ is $4$.
Observe that a length $2$ partition $\tau=ab$ is triangular, if and only if $2b\leq a+1$. 
 \Yvcentermath0

For $\tau$ a given triangular partition, and any positive integer $n$ larger or equal to $ \length(\tau)+\maxmult(\tau)$, we consider expressions
\begin{equation}\label{E_tau_qx}
     \E_{\tau}^{(n)}(\boldsymbol{q};\boldsymbol{x})=\sum_{\mu\vdash n} \sum_{\lambda} a_{\lambda\mu}\, s_\lambda(\boldsymbol{q}) s_\mu(\boldsymbol{x}),
 \end{equation}
where the coefficients $a_{\lambda\mu}$ are positive integers,  $\boldsymbol{q}$ stands for an ``alphabet'' of $k$ variables $(q_1,q_2,\ldots,q_k)$, and $\boldsymbol{x}=(x_1,x_2,x_3,\,\ldots\,)$. The $ \E_{\tau}^{(n)}$ are intended to be many variables extensions of the two variable context  $\E_{\tau}^{(n)}(q,t;\boldsymbol{x})$ (hence $q_1=q$ and $q_2=t$) described in see~\cite{AnyLine}. For any given $\tau$, these two variable instances may be calculated either using operators of Negut in the Schimann-Burban elliptic Hall algebra, or via a combinatorial formula expressed in terms of partitions contained in $\tau$ and involving ``area'' and ``dinv'' parameters, as well as LLT polynomials. See {\sl loc. cit.} for details, as well as \cite{MR4238181, MR3787405, MR4054520, MR4348234} for more context on the whole subject.  The $\E_{\tau}^{(n)}(q,t;\boldsymbol{x})$ are symmetric in $q$ and $t$, and ``Schur positive'' both as functions of $q$ and $t$, and the $\boldsymbol{x}$-variables. Hence, \autoref{E_tau_qx} makes sense. We recall that
\begin{equation}\label{nabla_prop}
	\nabla(\E_{\tau}^{(n)}(q,t;\boldsymbol{x})) = \E_{\tau+\delta(n)}^{(n)}(q,t;\boldsymbol{x}),
\end{equation}
where addition of partitions is part-wise. Here, $\nabla$ is the Macdonald eigenoperator with respective eigenvaluer equal to $\prod_{(i,j)\in \mu} q^it^j$ for the eigenfunctionr $\widetilde{H}_\mu(q,t;\boldsymbol{x})$. These are the modified Macdonald polynomials  (see~\cite{MR3682396}). In the sequel, we will describe further properties of the $\E_{\tau}^{(n)}$ in relation with the more general Macdonald eigenoperators $\Delta_f$ (see~\cite{MR1803316, MR3811519}), whose eigenvalues are $f(\sum_{(i,j)\in \mu} q^it^j)$. In plethystic notation, this means that the monomials $q^it^j$ are substituted for the variables in the symmetric function $f(z_1,z_2,\ldots, z_r)$, assuming that $r$ is equal to the number of cells $(i,j)$ of $\mu$.
   
Since in many instances the variable alphabets $\boldsymbol{q}$ and $\boldsymbol{x}$ are not actually used, it is handy to present $\E_{\tau}^{(n)}(\boldsymbol{q};\boldsymbol{x})$ as an abstract tensorial\footnote{We here exploit the non-commutativity of tensors, as well as their bilinearity.} expression, in which variables are removed (including notation-wise):
 \begin{equation}\label{defn_E_tau}
     \E_{\tau}^{(n)}:=\sum_{\mu\vdash n} \sum_{\lambda} a_{\lambda\mu}\, s_\lambda \otimes s_\mu. 
 \end{equation}
On such expressions we consider the Hall scalar product (corresponding to the $\boldsymbol{x}$-variables component) such that $\langle f\otimes s_\mu,s_\mu\rangle = f$, and $\langle f\otimes s_\mu,s_\nu\rangle = 0$ if $\nu\neq \mu$. In particular, $\langle \E_{\tau}^{(n)},s_\mu\rangle$ picks up the coefficient of $s_\mu$ in $\E_{\tau}^{(n)}$:
\begin{equation}
     \langle \E_{\tau}^{(n)},s_{\mu}\rangle = \sum_{\lambda} a_{\lambda\mu}\, s_\lambda.
\end{equation}
As we will see, some of these coefficients of $\E_{\tau}^{(n)}$ may depend on $n$, but they all ``stabilize'' when $n$ is large enough. To better express this, we consider 
$$\overline{\E}_{\tau}^{(n)}:=\sum_{\mu\vdash n} \sum_{\lambda} a_{\lambda\mu}\, 
	s_\lambda \otimes s_{\rouge{\overline{\mu}}},$$
where $ \overline{\mu}$ is the partition obtained by removing from $\mu$ its first column. For sure, we may get back $\E_{\tau}^{(n)}$ by adding a column to all partitions involved, so that the are turned back into partitions of $n$.
When $n$ becomes larger than some value $n_0$, there is a \define{stable} $\overline{\E}_{\tau}$ equal to $\overline{\E}_{\tau}^{(n)}$.
For the special case $\mu=1^n$, we set $\A_\tau^{(n)}:= \langle \E_{\tau}^{(n)}, s_{1^n}\rangle$.
As a matter of fact, $\A_{\tau}^{(n)}$ is independent of $n$, so that we will denote it more simply by $\A_\tau$. 

In all that follows, when applied to a tensorial expression, the linear operator $e_k^\perp$ acts on first components of tensors, meaning that
    $$ e_k^\perp (f\otimes g) := (e_k^\perp f)\otimes g.$$
On symmetric functions, $e_k^\perp$ is dual to multiplication by the elementary symmetric function $e_k$ for the Hall scalar product. In particular, its effect on a Schur function $s_\mu$ is:
    $$ e_k^\perp s_\mu = \sum_{\nu\subset_k \mu} s_\nu,$$
 where the sum runs over all partitions $\nu$ that can obtained by removing $k$ cells from $\mu$, no two of which lying in the same row.

It will be useful to consider that the \define{length} $\length(f)$ of a symmetric function $f$ is the maximum number of parts occurring in partitions $\lambda$ such that $\langle f,s_\lambda\rangle\neq 0$. In formula: 
    $$\length(f) :=\max_{\langle f,s_\lambda\rangle\not=0} \length(\lambda).$$
We extend this to tensors, setting $\length(f\otimes g):=\length(f)$.
As a Schur function vanishes whenever its number of variables is lower than the length of the indexing partition; fixing $k$, the number of variables $q_i$, corresponds to restrictions on lengths for left-hand components. We thus write
 \begin{equation}\label{defn_E_tau_restrict}
     \E_{\tau}^{(n)}\big|_{\rouge{\leq j}}:=\sum_{\mu\vdash n}\Big( \sum_{\length(\lambda)\rouge{\leq j}} a_{\lambda\mu}\, s_\lambda\Big) \otimes s_\mu,
 \end{equation}
 for the \define{restriction} of $\E_{\tau}^{(n)}$ to its terms of length at most $j$, in other terms all those for which $\length(s_\lambda \otimes s_\mu)\leq j$. In particular, $\E_{\tau}^{(n)}\big|_{\leq 2}$ bijectively encodes $\E_{\tau}^{(n)}(q,t;\boldsymbol{x})$. 
It may be shown that $\E_{\tau}^{(n)}\big|_{\leq 1}$   is entirely characterized  (up to a power of $q$) by a Whittaker\footnote{The  Whittaker polynomial $W_\mu(q;\boldsymbol{x})$ occurs as the highest $t$-degree component of the modified Macdonald polynomials $\widetilde{H}_\mu(q,t;\boldsymbol{x})$.} polynomial. Indeed, we have the following. 
  \begin{proposition}\label{q_specialization} For any triangular partition $\tau$ and $n\geq \length(\tau)+\maxmult(\tau)$, 
     \begin{equation}
         \E_{\tau}^{(n)}(q;\boldsymbol{x})= q^{|\tau|-\eta(\mu)} W_{\mu}(q;\boldsymbol{x}),
      \end{equation}
 where $\mu=\mu(\tau,n)$ is the conjugate of the partition $(n-\length(\tau),\rho_1,\ldots,\rho_k)$, in which the $\rho_i$'s are the multiplicities (reordered decreasingly) of parts occurring in $\tau$. Furthermore, $\eta(\mu)$ stands for $\sum_{(i,j)\in \mu}j$.
  \end{proposition}
To see how this describes the length-one component of $\E_{\tau}^{(n)}$, consider the case $\tau=321$, and $n=6$, for which we get
\begin{tcolorbox}$\begin{aligned}
 	&W_{41}(q,\boldsymbol{x})= s_{411} 
		+ (q+q^{2} ) s_{321} 
		+ q^{3} s_{222} 
		+ ( q+ q^{2} + q^{3}) s_{3111}\\
	  &\qquad \qquad \qquad
	  	+ (q^{2}+ q^{3} + q^{4} ) s_{2211} 
	  	+ (q^{3}+ q^{4} + q^{5} ) s_{21111} 
		+ q^{6} s_{111111},\qquad {\rm and}\\
 	&\E_{321}^{(6)}\big|_{\leq 1}=1 \otimes s_{411}
		+ (s_{1} + s_{2}) \otimes s_{321}  
		+ s_{3} \otimes s_{222}   
		+ (s_{1}  + s_{2}  + s_{3}) \otimes s_{3111}\\
	 &\qquad \qquad \qquad
		+ (s_{2}   + s_{3} + s_{4} )\otimes s_{2211} 
		+ (s_{3}   + s_{4}  + s_{5}) \otimes s_{21111} 
		+ s_{6} \otimes s_{111111}.
\end{aligned}$\end{tcolorbox}
In view of the special role of the restriction to ``two parameters'', we set:
  \begin{equation}\label{defn_epsilon}
     \varepsilon_{\tau}^{(n)}:=\E_{\tau}^{(n)}\big|_{\leq 2},
 \end{equation}
Recall that the $\varepsilon_{\tau}^{(n)}$ may be explicitly calculated, either combinatorially or using operators on symmetric functions, as described in \cite{AnyLine}. The crucial role of \autoref{nabla_prop} is made more apparent if we write all $\nabla(f)$ in tensor form. 
For sure, in all cases, specialization at $t=0$ is given by \autoref{q_specialization}.
Our main objective is to describe ``good'' many valued extensions $\E_{\tau}^{(n)}$ of the $\varepsilon_{\tau}^{(n)}$.
The precise meaning of ``good'' will be discussed in the sequel. It relies on a set of desired properties of the $\E_{\tau}^{(n)}$.

\section{\bleu{Some explicit values for \texorpdfstring{$\A_\tau$}{AT} and \texorpdfstring{$\E_\tau$}{ET}}}
For the sequel, we assume that the following holds.
 \begin{conjecture}[Coefficient-Length]\label{length_conjecture} For all triangular partitions $\tau$, we have 
 \begin{equation}
 \length(\E_{\tau}^{(n)})\leq \min(\length(\tau),\length(\tau')),
 \end{equation}
and $\A_{\tau}=\langle \E_{\tau}^{(n)},e_n\rangle$ does not depend on $n$. 
\end{conjecture}
It follows that, $\E_{\tau}^{(n)}$ and $\varepsilon_{\tau}^{(n)}$ must coincide whenever $\min(\length(\tau),\length(\tau'))\leq 2$.  
For all $n$, the cases with at most one part, or having all parts equal to $1$,  are all covered by the stable formulas
\begin{align}
            &\E_{0}^{(n)}=1\otimes  s_{1^n}, &&\BarE_{0} =1\otimes  1;\label{E0_3}\\
            &\E_{d}^{(n)}=s_{d-1}\otimes s_{21^{n-2}} + s_d\otimes s_{1^n}, &&\BarE_{d}=s_{d-1}\otimes s_{1} + s_d\otimes 1;\label{Ed_3}
            \intertext{and}
             &\E_{1^d}^{(n)}=\sum_{k=0}^d s_{d-k}\otimes s_{2^k1^{n-2k}}, &&\BarE_{1^d}=\sum_{k=0}^d s_{d-k}\otimes s_{1^k}.\label{E1d_3}
 \end{align}
In each instance, these formulas hold for any $n\geq \length(\tau)+\maxmult(\tau)$. They are thus stable with no further restriction. We can thus present them respectively in the format $\BarE_{\tau}$, as shown above,
without loss of information.

For two-row triangular partitions, we have the following explicit values when $n=\length(\tau)+\maxmult(\tau)$. Notice that this forces $n=3$. 
\begin{proposition}\label{Prop_cas_n3} For any two row triangular partition $\tau=(a,b)$, hence with $2b\leq a+1$, we have 
 \begin{equation}\label{formule _A_ab}
   \A_{\tau}= \sum_{\substack{0\leq d\leq b\\
                                                 3d\leq a+b}} s_{(a+b-2d,d)}.
 \end{equation}
 Moreover,
 \begin{equation}\label{formule _E_ab_3}
   \E_{\tau}^{(3)}= (e_2^\perp  \A_{\tau})\otimes s_3+ (e_1^\perp  \A_{\tau})\otimes s_{21} + \A_\tau\otimes s_{111}.
 \end{equation}	
 Furthermore $e_2^\perp  \A_{\tau+\delta(3)} =  \A_{\tau}$, which implies that $e_2^\perp \E_{\tau+\delta(3)}^{(3)} =\E_{\tau}^{(3)}$.
 \end{proposition}
 \begin{proof}[\bleu{\bf Proof}]
 The last implication follows directly from \autoref{formule _E_ab_3}. Indeed, assume $e_2^\perp  \A_{\tau+\delta(3)} =  \A_{\tau}$, and  apply $e_2^\perp$ to both sides the $\tau+\delta(3)$ instance of \autoref{formule _E_ab_3}. Then, we do get
 \begin{align*}
   e_2^\perp\E_{\tau+\delta(3)}^{(3)}&= e_2^\perp(e_2^\perp  \A_{\tau+\delta(3)})\otimes s_3+ e_2^\perp(e_1^\perp  \A_{\tau+\delta(3)})\otimes s_{21} + e_2^\perp\A_{\tau+\delta(3)}\otimes s_{111}\\
   &= e_2^\perp(\A_{\tau})\otimes s_3+ e_1^\perp(\A_{\tau})\otimes s_{21} + \A_{\tau}\otimes s_{111} = \E_{\tau}^{(3)},  
   \end{align*}
 since the operator $e_2^\perp$ and $e_1^\perp$ commute. Recall that 
 	$$e_2^\perp s_{(c,d)} =\begin{cases}
     s_{(c-1,d-1)}  & \text{if}\ d>0, \\
     0& \text{if}\ d=0.
\end{cases}$$
Thus, assuming \autoref {formule _A_ab} with $\tau+\delta(3) = (a+2,b+1)$, we readily calculate that
$$
 e_2^\perp  \A_{\tau+\delta(3)} =\sum_{\substack{1\leq d\leq b+1\\3d\leq a+b+3}}  e_2^\perp s_{(a+b+3-2d,d)}
                                                 =\sum_{\substack{1\leq d\leq b+1\\3d\leq a+b+3}}  s_{(a+b+2-2d,d-1)}
                                                 =\sum_{\substack{0\leq d\leq b\\3d\leq a+b}}  s_{(a+b-2d,d)}= \A_\tau,
$$
replacing $d-1$ by $d$ for the summation in the middle. 

We  recursively show that both \autoref{formule _E_ab_3} and \autoref{formule _A_ab} hold. Observe that, in tensor product terms, we have
  \begin{align}
       &  \nabla(s_3) = s_{22} \otimes s_{21} + s_{32}\otimes s_{111},\\
       & \nabla(s_{21}) =- (s_{21}\otimes s_{21} + s_{31}\otimes s_{111}),\\
       &  \nabla(s_{111})  = 1\otimes s_3 +(s_1+s_2)\otimes s_{21} + (s_3+s_{11})\otimes s_{111}.
    \end{align}
Hence
\begin{align*}
   \E_{\tau+\delta(3)}^{(3)}&= \nabla(\E_{\tau}^{(3)})\\
   &=\nabla(   (e_2^\perp  \A_{\tau})\otimes s_3+ (e_1^\perp  \A_{\tau})\otimes s_{21} + \A_\tau\otimes s_{111})\\
   &=  \A_\tau \otimes s_{3} \\
	&\qquad  + (s_{22} \cdot (e_2^\perp  \A_{\tau})- s_{21}\cdot (e_1^\perp  \A_{\tau})+(s_1+s_2)\cdot  \A_\tau) \otimes s_{21} \\
	&\qquad + (s_{32}\cdot (e_2^\perp  \A_{\tau}) - s_{31}\cdot (e_1^\perp  \A_{\tau}) +(s_3+s_{11})\cdot  \A_\tau)\otimes s_{111}
   \end{align*}
where all calculations are restricted to length two.  We have already checked  that $\A_\tau =  e_2^\perp  \A_{\tau+\delta(3)} $ is compatible with \autoref{formule _A_ab}. Let us see that the coefficients of $s_{111}$ in $\E_{\tau+\delta(3)}^{(3)}$ agrees with the right-hand side of the above formula, namely:
  $$\A_{\tau+\delta(3)} =s_{32}\cdot (e_2^\perp  \A_{\tau}) - s_{31}\cdot (e_1^\perp  \A_{\tau}) +(s_3+s_{11})\cdot  \A_\tau.$$
 To this end, consider any term $s_{(a+b+3-2d,d)}$ occurring in the right-hand side of \autoref{formule _A_ab} for $\tau=(a+2,b+1)$, hence with $0\leq d\leq b+1$ and $3d\leq a+b+3$. We breakup the scalar product:
    $$\langle s_{(a+b+3-2d,d)},s_{32}\cdot (e_2^\perp  \A_{\tau}) - s_{31}\cdot (e_1^\perp  \A_{\tau}) +(s_3+s_{11})\cdot  \A_\tau\rangle,$$
 into its three components, and use the duality of multiplication by $g$ and $g^\perp$ to get the evaluations
\begin{align*}
    \langle s_{(a+b+3-2d,d)},s_{32}\cdot (e_2^\perp  \A_{\tau})\rangle &=  \langle e_2\, s_{32}^\perp\, s_{(a+b+3-2d,d)},\A_{\tau}\rangle\\
				&=  \langle s_{(a+b+1-2d,d-1)},\A_{\tau}\rangle,\\
    \langle s_{(a+b+3-2d,d)},-s_{31}\cdot (e_1^\perp  \A_{\tau}) \rangle &= -\langle e_1\, s_{31}^\perp\,  s_{(a+b+3-2d,d)},  \A_{\tau} \rangle \\
    				&= -\langle s_{(a+b+1-2d,d-1)}+s_{(a+b-2d,d)},  \A_{\tau} \rangle, \\
    \langle s_{(a+b+3-2d,d)}, (s_3+s_{11})\cdot  \A_\tau \rangle &=\langle (s_3+s_{11})^\perp\,s_{(a+b+3-2d,d)}, \A_\tau \rangle\\
    				&=\langle (s_3)^\perp\,s_{(a+b+3-2d,d)}+s_{(a+b+2(d-1),d-1)}, \A_\tau \rangle.
 \end{align*}
To evaluate these scalar product, we now assume by recurrence that $\A_\tau$ is given by \autoref{formule _A_ab}, hence it only involves terms of the form $s_{(a+b-2d',d')}$, this reduces our calculation to 
  $$\langle s_{(a+b+3-2d,d)},\A_{\tau+\delta(3)}\rangle = \langle s_{(a+b-2(d-1),d-1)},\A_{\tau+\delta(3)}\rangle=1.$$
 The other term corresponds to an entirely similarly verification that $e_1^\perp \A_{\tau+\delta(3)}=(s_{22} \cdot (e_2^\perp  \A_{\tau})- s_{21}\cdot (e_1^\perp  \A_{\tau})+(s_1+s_2)\cdot  \A_\tau)$. 
 \end{proof}
The following explicit values\footnote{Calculated in part using Sagemath code made available by the authors of~\cite{AnyLine}.} for  $\E_{\tau}^{(n)}=\varepsilon_{\tau}^{(n)}$, for $n\geq \length(\tau)+\maxmult(\tau)$, illustrate that
	$$\overline{ \E}_{\tau}^{(n+1)} -\overline{ \E}_{\tau}^{(n)},$$
 is Schur positive\footnote{With positive integer coefficients for $s_\lambda\otimes s_\mu$ terms.}. Clearly stability means that this difference vanishes whenever $n$ is large enough. For instance, for $b=1$ \autoref{formule _E_ab_3} becomes stable with:
\begin{equation}\label{cas_a1}
	\overline{\E}_{(a,1)}=\overline{\E}^{(3)}_{(a,1)} + \rouge{s_{a-1}\otimes s_{11}},\qquad \rouge{n\geq 4}.	    
  \end{equation}
Further examples of successive increments up to the stable value, are the following:
\begin{tcolorbox}$\begin{aligned}
\BarE_{211}^{(4)}&=(s_4+s_{21})\otimes 1+ (s_2+s_3+s_{11})\otimes s_{1} + s_1\otimes s_{2} + s_2 \otimes s_{11},\hfill{}\\  
\BarE_{211}^{(5)}&=\overline{\E}_{211}^{(4)} +s_1\otimes s_{11}+1\otimes s_{21},\\
\BarE_{211}&=\overline{\E}_{211}^{(5)} +s_1\otimes s_{111};
\end{aligned}$\end{tcolorbox}

\begin{tcolorbox}$\begin{aligned}
\BarE_{ 32 }^{(3)}&=( s_{5} + s_{31} ) \otimes 1 
	+ ( s_{3} + s_{4} + s_{21} ) \otimes s_{1}
	+ s_{2} \otimes s_{2}\\
\BarE_{ 32 }&=\overline{\E}_{ 32 }^{(3)}+ (s_{3}+s_{11}) \otimes s_{11};
\end{aligned}$\end{tcolorbox}

\begin{tcolorbox}$\begin{aligned}
\BarE_{ 221 }^{(4)}&=( s_{5} + s_{31} ) \otimes 1 
	+ ( s_{3} + s_{4} + s_{21} ) \otimes s_{1}
	+ s_{2} \otimes s_{2}
	+ ( s_{3} + s_{11} ) \otimes s_{11},\\
\BarE_{ 221 }^{(5)}&=\BarE_{ 221 }^{(4)}	+ s_2 \otimes s_{11}	+ s_{1} \otimes s_{21},\\
\BarE_{ 221 }&=\BarE_{ 221 }^{(5)}	+ s_2 \otimes s_{111};
\end{aligned}$\end{tcolorbox}

\begin{tcolorbox}$\begin{aligned}
\BarE_{ 2111 }^{(6)}&=( s_{5} + s_{31} ) \otimes 1 
	+ ( s_{3} + s_{4} + s_{21} ) \otimes s_{1}
	+ s_{2} \otimes s_{2}
	+  ( s_{2} + s_{3} + s_{11} ) \otimes s_{11},\\
	&\qquad + s_1 \otimes s_{21}+s_2 \otimes s_{111},\\
\BarE_{ 2111 }^{(7)}&=\BarE_{ 2111 }^{(6)}+1 \otimes s_{211}+ s_2\otimes s_{111}\\
\BarE_{ 2111 }&=\BarE_{ 2111 }^{(7)}+ s_1 \otimes s_{1111}.
\end{aligned}$\end{tcolorbox}

\section{\bleu{Observed Properties}}
Among observed property are the following. First, we recall that $\A_\tau$ is independent form $n$ (assumed at least to be larger than $\length(\tau)+1$).
 \begin{property}[Conjugation]\label{conjugation} For any $\tau$ and $n$, the respective alternating components of $\E_\tau^{(n)}$ and $\E_{\tau'}^{(n)}$ are equal. Namely, $\A_\tau = \A_{\tau'}$.
\end{property}
Small values of $\A_\tau$ are as follows (for dominant triangular partitions).

\begin{small}
\begin{tcolorbox}$\begin{aligned}
&\A_{0}=1, && \A_{1}=s_{1},&& \A_{2}=s_{2},&&\A_{3}=s_{3},&& \ldots\\
&\A_{21}=s_{3} + s_{11}, && \A_{31}=s_{4} + s_{21},&& \A_{41}=s_{5} + s_{31},&&\A_{51}=s_{6} + s_{41},&& \ldots\\
&\A_{32}=s_{5} + s_{31}, && \A_{42}=s_{6} + s_{41} + s_{22}, && \A_{52}=s_{7} + s_{51} + s_{32},&& \A_{62}=s_{8} + s_{61} + s_{42},&& \ldots
\end{aligned}$\end{tcolorbox}

\begin{tcolorbox}$\begin{aligned}
&\A_{53}=s_{8} + s_{61} + s_{42}, 
&&\A_{63}=s_{9} + s_{71} + s_{52} + s_{33}, &
&\A_{73}=s_{10.} + s_{81} + s_{62} + s_{43}, \\
&\A_{83}=s_{11.} + s_{72} + s_{53} + s_{91}, &
&\A_{93}=s_{12.} + s_{82} + s_{63} + s_{10.1}, && \ldots
\end{aligned}$\end{tcolorbox}

\begin{tcolorbox}$\begin{aligned}
&\A_{321}=s_{6} + s_{31} + s_{41} + s_{111}, \\
&\A_{421}=s_{7} + s_{32} + s_{41} + s_{51} + s_{211},\\
&\A_{431}=s_{8} + s_{42} + s_{51} + s_{61} + s_{311},\\
&\A_{531}=s_{9} + s_{71} + s_{61} + s_{52} + s_{42} + s_{411} + s_{221},\\
&\A_{631}=s_{10.} + s_{81} + s_{62} + s_{43} + s_{52} + s_{71} + s_{511} + s_{321}; 
\end{aligned}$\end{tcolorbox}

\begin{tcolorbox}$\begin{aligned}
&\A_{432}=s_{9} + s_{33} + s_{52} + s_{61} + s_{71} + s_{411},\\
&\A_{532}=s_{10.} + s_{81} + s_{62} + s_{43} + s_{52} + s_{71} + s_{511} + s_{321}, \\
&\A_{641}=s_{11.} + s_{91} + s_{72} + s_{53} + s_{62} + s_{81} + s_{611} + s_{421}, \\
&\A_{741}=s_{12.} + s_{10.1} + s_{72} + s_{91} + s_{82} + s_{63} + s_{53} + s_{331} + s_{521} + s_{711},\\
&\A_{841}=s_{13.} + s_{92} + s_{73} + s_{54} + s_{63} + s_{10.1} + s_{11.1} + s_{82} + s_{431} + s_{621} + s_{811}, \\
&\A_{542}=s_{11.} + s_{91} + s_{72} + s_{53} + s_{62} + s_{81} + s_{611} + s_{421}, \\
&\A_{642}=s_{12.} + s_{10.1} + s_{91} + s_{82} + s_{72} + s_{63} + s_{62} + s_{44} + s_{222} + s_{421} + s_{521} + s_{711}, \\
&\A_{742}=s_{13.} + s_{92} + s_{10.1} + s_{73} + s_{82} + s_{54} + s_{63} + s_{72} + s_{11.1} + s_{322} + s_{431} + s_{521} + s_{621} + s_{811};
\end{aligned}$\end{tcolorbox}

\begin{tcolorbox}$\begin{aligned}
&\A_{4321}=s_{10.} + s_{42} + s_{43} + s_{61} + s_{62} + s_{71} + s_{81} + s_{311} + s_{411} + s_{511} + s_{1111}, \\
&\A_{5321}=s_{11.} + s_{43} + s_{52} + s_{53} + s_{62} + s_{71} + s_{72} + s_{81} + s_{91} \\
   &\qquad\qquad + s_{321} + s_{411} + s_{421} + s_{511} + s_{611} + s_{2111}, \\
&\A_{5421}=s_{12.} + s_{44} + s_{53} + s_{62} + s_{63} + s_{72} + s_{81} + s_{82} + s_{91} + s_{10.1} \\
   &\qquad\qquad + s_{331} + s_{421} + s_{511} + s_{521} + s_{611} + s_{711} + s_{3111}, \\
&\A_{5431}=s_{13.} + s_{54} + s_{63} + s_{72} + s_{73} + s_{82} + s_{91} + s_{92} + s_{10.1} + s_{11.1} \\
   &\qquad\qquad + s_{431} + s_{521} + s_{611} + s_{621} + s_{711} + s_{811} + s_{4111}, \\
&\A_{6421}=s_{13.} + s_{53} + s_{54} + s_{63} + 2 s_{72} + s_{73} + s_{82} + s_{91} + s_{92} + s_{10.1} + s_{11.1} \\
   &\qquad\qquad + s_{421} + s_{431} + 2 s_{521} + s_{611} + s_{621} + s_{711} + s_{811} + s_{2211} + s_{4111}.
\end{aligned}$\end{tcolorbox}
\end{small}

To state our next observation, we use \define{Frobenius's notation} for hook shape partitions, writing  $(a\,|\,\ell )$ for the partition $(a+1,1^\ell)$ (see Figure~\ref{Fig1}). In other terms, $a$ is the \define{arm} of the hook, and $\ell$ is its \define{leg}. 
\begin{figure}[h]
\begin{center}
 \begin{tikzpicture}[scale=.6]
 \Yboxdim{1cm}
 \Yfillcolour{azure!80}
  	 \tyng(0cm,1cm,1,1,1);
\Yfillcolour{red!60}
  	 \tyng(1cm,0cm,5);
\Yfillcolour{yellow}
	\tyng(0cm,0cm,1);
	\node at (-.5,2.5) {$\ell\left.\rule{0cm}{28pt}\right\{$};
	\node at (3.5,-0.5) {$\underbrace{\hskip3cm}_{\textstyle a}$};
 \end{tikzpicture}
 \end{center}
    \vskip-20pt
\caption{The hook shape $(a\,|\,\ell )$.}\label{Fig1}
\end{figure}
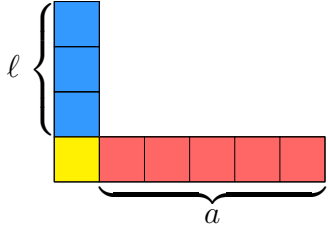
Recall  that the operator $e_k^\perp$ is the dual of multiplication by $e_k$ for the Hall scalar product, which effect is well known to be given by the Pieri rule.

\begin{property}[Hook-Components]\label{skew_hook} For any size $n$ hook shape $(a\,|\,\ell)$, we have the identity
   \begin{equation}
       e_a^\perp \A_\tau = \langle \E_\tau^{(n)}, s_{(a\,|\,\ell)}\rangle,
   \end{equation}
  hence the coefficient of $s_{(a\,|\,\ell)}$ in $\E_\tau^{(n)}$ is independent of $n$.
\end{property}
Recall that, for two single variables $u$ and $v$, the plethysm $p_k[u-\varepsilon v]$ is defined to be $(u^k-(-1)^k\,v^k)$. As usual this is extended linearly and multiplicatively to symmetric functions expressed in the power-sum basis. One checks that 
\begin{equation}
\frac{1}{u+v} s_\mu[u-\varepsilon v]:=\begin{cases}
     u^av^\ell & \text{if}\ \mu=(a\,|\,\ell), \\
      0 & \text{otherwise}.
\end{cases}
\end{equation}
In other terms, the plethysm considered in the left hand side of the above establishes a direct bijective correspondance 
   $$s_{(a\,|\,\ell)}\leftrightarrow u^av^\ell,$$
 between the hook portion of a linear combinations of Schur functions, and polynomials in $u$ and $v$. Thus the following proposition entirely characterizes the hook portion of $\A_\tau$. It follows from \autoref{skew_hook} and \autoref{q_specialization}. 

 \begin{proposition}[Hook-alternants]
 For any triangular partition $\tau$, we have
\begin{equation}\label{formula_hooks_en}
   \frac{1}{u+v}\A_\tau[u-\varepsilon v]=u^{|\tau|-\binom{k+1}{2}}(u^2+v)\cdots (u^{m}+v),
 \end{equation}
 where $m=\min(\length(\tau),\length(\tau'))$.
 \end{proposition}
For instance, the hook terms in $\A_{54321}$ are:
\begin{align*}
 \A_{54321}\big|_{\rm hooks} &= s_{(14\,|\,0)} + s_{(12\,|\,1)} + s_{(11\,|\,1)} + s_{(10\,|\,1)} + s_{(9\,|\,2)} + s_{(9\,|\,1)}  + s_{(8\,|\,2)} + 2 s_{(7\,|\,2)} \\
   &\qquad\qquad + s_{(6\,|\,2)} + s_{(5\,|\,3)} + s_{(5\,|\,2)} +  s_{(4\,|\,3)} + s_{(3\,|\,3)} + s_{(2\,|\,3)} + s_{(0\,|\,4)},
 \end{align*}
and the corresponding encoding polynomial  is:
\begin{align*} (u^2+v) (u^3+v) (u^4+v) &(u^5+v) \\
&= u^{14} + u^{12} v + u^{11} v + u^{10} v + u^{9} v^{2} + u^{9} v + u^{8} v^{2} + 2 u^{7} v^{2}\\
    &\qquad\qquad  + u^{6} v^{2} + u^{5} v^{3} + u^{5} v^{2} + u^{4} v^{3} + u^{3} v^{3} + u^{2} v^{3} + v^{4}.
 \end{align*}
Further calculations suggest that \autoref{Prop_cas_n3} should extend to all cases. Thus we have:
 \begin{property}[Skew version of $\nabla^{-1}$] Letting $m=\min(\length(\tau),\length(\tau'))$, for any triangular partition $\tau$, then
\begin{equation}
 	  e_m^\perp  \E_{\tau+\delta(m)}^{(n)} =  \E_{\tau}^{(n)},
\end{equation}
for all $n\geq \length(\tau)+\maxmult(\tau)$.
 \end{property}
The following more general property holds.
\begin{property}[Delta Property]\label{DeltaProperty}
For a dominant triangular partition $\tau$, with  $n:=\length(\tau)+1$ and such that $\tau-\delta(n)$ is triangular, we have
\begin{equation}\label{DeltaGeneral}
 \Big(e_k^\perp \E_{\tau}^{(n)}\Big)\Big|_{\leq 2}  =\Delta'_{e_{n-k-1}}  \Big( \E_{\tau-\delta(n)}^{(n)}\Big|_{\leq 2} \Big)
\end{equation}
 for any $0\leq k\leq n-1$.
\end{property}
\noindent
When $n=\length(\tau)+1$, \autoref{DeltaGeneral} uniquely characterizes $\E_{\tau}^{(n)}$ for all $\tau$ dominant of length at most $5$.


\section{\bleu{\texorpdfstring{$\nabla$}{N} of hook indexed Schur}}
Among the special instances for which we have a representation theoretic description of $\E_\tau^{(n)}$ (in terms of $\GL_k\times \S_n$-modules), an interesting family corresponds to the case $\nabla(\SchurHat_{(a\,|\,\ell)})$, where we set
\begin{formula}\label{nabla_formula}
   \SchurHat_{(a\,|\,\ell)}(\boldsymbol{x}):= (-1/{qt})^{a-1} s_{(a\,|\,\ell)}(\boldsymbol{x}). 
\end{formula}
For more on this, see~\cite{NablaHook,2003.07402}. The interesting observation is that $\nabla$-image of these renormalize hook indexed Schur functions corresponds to $\varepsilon_{\tau}^{(n)}$. More precisely,
\begin{property}\label{nabla_hook}
If $\delta(n-1) \subseteq \tau \subseteq \delta(n)$, such that $\length(\tau')\preceq  \length(\tau)$ ($\tau'$ is dominant), then
\begin{equation}
   \Big( \E_\tau^{(n)}\Big)\Big|_{\leq 2} =   \nabla(\SchurHat_{(a\,|\,\ell)}),
\end{equation}
with $a=\binom{n+1}{2}-|\tau|$ and $\ell=n-a-1$.
\end{property}
 \Yvcentermath1
For instance, with $n=5$, the relevant partitions are
$$\hbox{\yng(3,2,1)}\subseteq \hbox{\yng(3,2,1,1)}\subseteq \hbox{\yng(3,2,2,1)} \subseteq \hbox{\yng(3,3,2,1)} \subseteq\hbox{\yng(4,3,2,1)}.$$
for which we have the following values 
\begin{tcolorbox}$\begin{aligned}
 \E_{321}^{(5)}&=\nabla(\SchurHat_{(4\,|\,0)})+s_{111}\otimes s_{11111},\\
 \E_{3211}^{(5)}&=\nabla(\SchurHat_{(3\,|\,1)})+s_{111}\otimes s_{2111}+s_{211}\otimes s_{2111},\\
\E_{3221}^{(5)}&=\nabla(\SchurHat_{(2\,|\,2)})+s_{111}\otimes s_{221}+s_{211}\otimes s_{2111}+s_{311}\otimes s_{11111},\\
\E_{3321}^{(5)}&=\nabla(\SchurHat_{(1\,|\,3)})+s_{211}\otimes s_{221}+s_{311}\otimes s_{2111}+s_{411}\otimes s_{11111},\\
\E_{4321}^{(5)}&=\nabla(\SchurHat_{(0\,|\,4)})+(s_{211}+s_{311})\otimes s_{221} +(s_{111}+s_{211}+s_{311}+s_{411})\otimes s_{2111}\\
           &\hskip2.2cm +(s_{311}+s_{411}+s_{511}+s_{1111})\otimes s_{11111}.\\
 \end{aligned}$\end{tcolorbox}
Written in terms of $e_k^\perp \A_\tau$ to make the expressions more compact, these expand as:
\begin{tcolorbox}$\begin{aligned}
\E_{321}^{(5)} &=( s_{6} + s_{31} + s_{41} + s_{111} ) \otimes s_{11111}\\
   &\qquad + ( s_{2} + s_{3} + s_{4} + s_{11} + s_{21} ) \otimes s_{221} + ( s_{1} + s_{2} ) \otimes s_{32} \\  
   &\qquad + (e_1^\perp \A_{321}) \otimes s_{2111}
		 + ( e_2^\perp \A_{321} )\otimes s_{311}
		 + ( e_3^\perp \A_{321} )\otimes s_{41},
\end{aligned}$\end{tcolorbox}

\begin{tcolorbox}$\begin{aligned}
\E_{3211}^{(5)} &=( s_{7} + s_{32} + s_{41} + s_{51} + s_{211} ) \otimes s_{11111}\\
   &\qquad + ( s_{3} + s_{4} + s_{5} + 2 s_{21} + s_{31} ) \otimes s_{221}
		+ ( s_{2} + s_{3} + s_{11} ) \otimes s_{32}\\
   &\qquad + (e_1^\perp \A_{3211}) \otimes s_{2111}
		+ ( e_2^\perp \A_{3211} )\otimes s_{311}
		+ ( e_3^\perp \A_{3211} )\otimes s_{41},
\end{aligned}$\end{tcolorbox}

\begin{tcolorbox}$\begin{aligned}
\E_{3221}^{(5)} &=( s_{8} + s_{42} + s_{51} + s_{61} + s_{311} )  \otimes s_{11111}\\
   &\qquad + ( s_{4} + s_{5} + s_{6} + s_{22} + 2 s_{31} + s_{41} + s_{111} )  \otimes s_{221}
		+ ( s_{3} + s_{4} + s_{21} )\otimes s_{32}\\
   &\qquad + (e_1^\perp \A_{3221}) \otimes s_{2111}
		+ ( e_2^\perp \A_{3221} )\otimes s_{311}
		+ ( e_3^\perp \A_{3221} )\otimes s_{41},
\end{aligned}$\end{tcolorbox}

\begin{tcolorbox}$\begin{aligned}
\E_{3321}^{(5)} &=( s_{9} + s_{33} + s_{52} + s_{61} + s_{71} + s_{411} )  \otimes s_{11111}\\
   &\qquad + ( s_{5} + s_{6} + s_{7} + s_{31} + 2 s_{41} + s_{51} + s_{211} + s_{32} )  \otimes s_{221}\\
   &\qquad + ( s_{4} + s_{5} + s_{21} + s_{31} )\otimes s_{32}\\
   &\qquad + (e_1^\perp \A_{3321}) \otimes s_{2111}
		+ ( e_2^\perp \A_{3321} )\otimes s_{311}
		+ ( e_3^\perp \A_{3321} )\otimes s_{41},
\end{aligned}$\end{tcolorbox}

\begin{tcolorbox}$\begin{aligned}
\E_{4321}^{(5)} &=
( s_{10.} + s_{42} + s_{43} + s_{61} + s_{62} + s_{71} + s_{81} + s_{311} + s_{411} + s_{511} + s_{1111} ) \otimes s_{11111}\\
   &\qquad + ( s_{4} + s_{5} + s_{6} + s_{7} + s_{8} + s_{21} + s_{31} + 2 s_{41} + 2 s_{51} + s_{61}\\
   &\qquad\qquad\qquad  + s_{22} + s_{32} + s_{42} + s_{211} + s_{311} ) \otimes s_{221}\\
   &\qquad + ( s_{2} + s_{3} + s_{4} + s_{5} + s_{6} + s_{21}  + s_{31} + s_{41} + s_{22}) \otimes s_{32}\\
   &\qquad + (e_1^\perp \A_{4321}) \otimes s_{2111}
		+ ( e_2^\perp \A_{4321} )\otimes s_{311}
		+ ( e_3^\perp \A_{4321} )\otimes s_{41}
		+ ( e_4^\perp \A_{4321} )\otimes s_{5}.
\end{aligned}$\end{tcolorbox}
Exploiting  \autoref{DeltaProperty}, \autoref{nabla_hook}, and the relation
  \begin{equation}\label{commutator}
     \Delta'_{e_{n-2}} (e_n) = \sum_{a=1}^{n-1} \nabla(\SchurHat_{(a\,|\,\ell)}),\qquad (\ell=n-1-a)
  \end{equation}
we infer that the above lifts to the following.
\begin{property}
For all $n$,
\begin{equation}
    e_1^\perp \E_{\delta(n)}^{(n)} = \sum_{\delta(n-1)\subseteq \tau\varsubsetneq \delta(n)}  \E_{\tau}^{(n)}.
\end{equation}
\end{property}
 \Yvcentermath0
 
\section{\bleu{Negut formula}}
A partition is said to be {\it triangular} if its cells are exactly the integral points of the convex hull of its diagram, and no cell of it lies in the convex hull its complement.   For all triangular partitions $\tau$ (called ``partition under any line'' in~\cite{AnyLine}) there is an associated $(q,t)$-coefficient symmetric function  which may be calculated by the following formula due to Negut:
$$\mathcal{E}^{(n)}_\tau(q,t;\boldsymbol{x}) =\sum_{\mu\vdash n}
     \sum_{\theta \in \mathrm{SYT}(\mu)} \Omega(\theta) \, \frac{T_\mu^\tau(\theta)}{T_\mu}\widetilde{H}_\mu(q,t;\boldsymbol{x}),$$
 where $n$ is some integer larger or equal to $\ell(\tau)+1$.  Moreover, for a standard tableau $\theta$ of shape $\mu$, we set 
 $$T_\mu^\tau(\theta):=\prod_{(i,j)\in \mu}(q^it^j)^{v_\tau(\theta(i,j))},$$ 
 with  $v_\tau:=(0,\tau_1-\tau_2,\ldots ,\tau_i-\tau_{i+1},\ldots)$, padding it with $0$ parts to make it a length-$n$ vector; and  
\begin{align*}
\Omega(\theta):= &\prod_{\theta(a,b)>\theta(i,j)}\frac{(q^a t^b-q^i t^j)^*
           (q^a t^b-q^{i+1}t^{j+1})^*}{
            (q^a t^b-q^{i+1}t^j)^*
             (q^a t^b-q^i t^{j+1})^*}\\
   &\qquad \times \prod_{\theta(a,b)=\theta(i,j)+1}\frac{q^a t^b}{q^a t^b-q^{i+1} t^{j+1}}
    \prod_{\theta(i,j)\neq 1} \frac{1}{q^i t^j-1}
 \end{align*} 
where $(A)^*$ is equal to $1$ if $A=0$, and $A$ otherwise. When $\tau$ is the staircase partition $(n-1,\ldots,2,1)$, we have $\mathcal{E}_\tau=\nabla(e_n)$, and other special cases include other symmetric functions associated rational and rectangular Catalan combinatorics using the techniques of the elliptic Hall algebra; as well as the effect of $\nabla$ on hook-shape Schur functions. 
\section{\bleu{Universal formula}}
An alternative description of the $e$-basis expansion of ${\E}_{\tau}(\boldsymbol{q}+1,\boldsymbol{x})$ is via the auxiliary expression\footnote{Recall that the plethysm $p_k(\boldsymbol{q}+1)=p_k(\boldsymbol{q})+1$ sends a Schur function $s_\rho$ to the sum    
\begin{equation}
   s_\rho(\boldsymbol{q}+1)=\sum_{\rho/\mu\ {\rm H.S.}} s_\mu(\boldsymbol{q}),
\end{equation} 
where $\mu$ runs over all partitions such that $\rho/\mu$ is an \define{horizontal strip} (H.S.). This is to say that no two cells of the skew shape $\mu/\rho$ lie in the same column. In our framework, we have the inverse plethysm equivalence: $f(\boldsymbol{q})=g(\boldsymbol{q}-1)$  iff $g(\boldsymbol{q})=f(\boldsymbol{q}+1)$.}
\begin{align}
	\F_{\tau}(\boldsymbol{q};\boldsymbol{x})&:=\downarrow\! \E_{\tau}(\boldsymbol{q}+1;\boldsymbol{x}) =\sum_{\lambda} \sum_{\mu} \alpha_{\mu\lambda}\, s_\mu(\boldsymbol{q})  (\downarrow\! e_\lambda(\boldsymbol{x}))\nonumber\\
	   &=\sum_{\mu,\nu} \beta_{\mu\nu}\, s_\mu(\boldsymbol{q})  e_\nu(\boldsymbol{x}),
\end{align}
with $(\downarrow\! e_\lambda)$ denoting the elementary function obtained by removing from $\lambda$ its first part.
We may present $\F_{\tau}$ in tensorial format as:
 \begin{equation}\label{defn_F_tau}
      \F_\tau=  \sum_{\mu,\nu} \beta_{\mu\nu}\, s_\mu \otimes  e_\nu. 
 \end{equation}
The advantage of $\F_\tau$ is that we can uniformly recuperate all $\E_\tau^{(n)}$, via the inverse plethysm
\begin{equation}\label{uniform_E}
   \E_\tau^{(n)}(\boldsymbol{q},\boldsymbol{x}) =  \uparrow^{(n)}\!\F_\tau(\boldsymbol{q}-1,\boldsymbol{x})
\end{equation}
where $\uparrow^{(n)}$ is the linear operator on $\boldsymbol{x}$-variable symmetric polynomials that send $e_\nu=e_\nu(\boldsymbol{x})$ to $(\uparrow^{(n)}\!  e_\nu):=e_{n-|\nu|}\,e_\nu$.
In other words, $(\uparrow^{(n)}\! e_\nu)$ is obtained by adding a part to $\nu$ so that the resulting partition has size $n$. 

For instance, we have
\begin{align}
   &\F_{0} = 1\otimes 1, \\
    &\F_{d} = s_d\otimes 1 + (1+s_1+\,\cdots\, + s_{d-1})\otimes e_1,  \\
  & \F_{1^d} = s_d\otimes 1 + s_{d-1}\otimes e_1 +\,\cdots\, + 1\otimes e_d\\
   &\F_{(a,1)} = (s_{a+1} + s_{(a-1,1)})\otimes 1\nonumber\\
   			&\qquad\qquad +(s_a+(1+s_2+\,\ldots\,+ s_{a-2})\cdot s_1)\otimes e_{1}\nonumber\\
			&\qquad\qquad + s_{a-1}\otimes e_2 +(1+s_1+\,\ldots\,+ s_{a-2})\otimes e_{11}.
\end{align}
Let us illustrate how \autoref{uniform_E} gives the correct description of both the following instances of \autoref{formule _E_ab_3} and \autoref{cas_a1}:
\begin{align}
   \E_{(a,1)}^{(3)} &=  s_{a-2}\otimes s_3+ (s_{a-1}+s_a+ s_{(a-2,1)})\otimes s_{21} + (s_{a+1}+s_{(a-1,1)})\otimes s_{111},\qquad {\rm and} \\
   \E_{(a,1)}^{(4)} &=s_{a-2}\otimes s_{31}+ (s_{a-1}+s_a+ s_{(a-2,1)})\otimes s_{211} + (s_{a+1}+s_{(a-1,1)})\otimes s_{1111}\nonumber\\
       &\qquad + s_{a-1}\otimes s_{22},
\end{align}
as
\begin{align*}
    \uparrow^{(n)}\!\F_{(a,1)}(\boldsymbol{q}-1,\boldsymbol{x}) &=
    (s_{a+1} -s_a-s_{a-1}+s_{a-2}+ s_{(a-1,1)}-s_{(a-2,1)})\otimes e_n\\
    &\qquad +(s_a-s_{a-2}+s_{(a-2,1)})\otimes e_{n-1}e_1\\
    &\qquad + (s_{a-1}-s_{a-2})\otimes e_{n-2}e_2 + s_{a-2}\otimes e_{n-2}e_{11}\displaybreak[0]\\
    &=(s_{a+1} -s_a-s_{a-1}+s_{a-2}+ s_{(a-1,1)}-s_{(a-2,1)})\otimes s_{1^n}\\
    &\qquad +(s_a-s_{a-2}+s_{(a-2,1)})\otimes (s_{21^{n-2}}+s_{1^n})\\
    &\qquad + (s_{a-1}-s_{a-2})\otimes (\rouge{\chi(n>3)}\,s_{221^{n-4}}+s_{21^{n-2}}+s_{1^n})\\
     &\qquad + s_{a-2}\otimes (s_{31^{n-3}} + \rouge{\chi(n>3)}\,s_{221^{n-4}} + 2 s_{21^{n-2}} + s_{1^{n}})
\end{align*}
where $\chi(n>3)=1$ if $n>3$, and  vanishes otherwise. Thus, the disparity between  $\E_{(a,1)}^{(3)}$ and the general value $\E_{(a,1)}^{(n)}$ only arise because of the differences between the product formulas:
\begin{align*}
     &e_1\cdot e_2 = s_{21} +s_{111},\\
     &e_1\cdot e_{11} = s_3+2s_{21} +s_{111},
\end{align*}
and 
\begin{align*}
      &e_n\cdot e_2 =s_{221^{n-4}}+ s_{21^{n-2}}+s_{1^n},\\
       &e_n\cdot e_{11} = s_{31^{n-3}} + s_{221^{n-4}} + 2 s_{21^{n-2}} + s_{1^{n}}
\end{align*}

More explicit values for $\F_\tau$ are as follows:
\begin{tcolorbox}$\begin{aligned}
 &\F_{ 211 }=&&( s_{4} + s_{21} ) \otimes 1\\
       &&&\qquad  + ( s_{2} + s_{3} + s_{11} ) \otimes e_{1}
	+ s_{2} \otimes e_{2}
	+ s_{1} \otimes e_{11}
	+ s_{1} \otimes e_{3}
	+1 \otimes e_{21};
\end{aligned}$\end{tcolorbox}

\begin{tcolorbox}$\begin{aligned}
&\F_{ 32 }=&&( s_{5} + s_{31} ) \otimes 1\\
    &&&\qquad
	+ ( s_{2} + s_{3} + s_{4} + s_{21} ) \otimes e_{1}
	+ ( s_{1} + s_{3} + s_{11} ) \otimes e_{2}
	+ (1 + s_{1} + s_{2} ) \otimes e_{11},\\
&\F_{ 221 }=&&( s_{5} + s_{31} ) \otimes 1\\
    &&&\qquad
	+ ( s_{3} + s_{4} + s_{21} ) \otimes e_{1}
	+ ( s_{1} + s_{3} + s_{11} ) \otimes e_{2}\\
    &&&\qquad
	+ s_{2} \otimes e_{11}
	+ s_{2} \otimes e_{3}
	+ ( 1+ s_{1} ) \otimes e_{21},\\
&\F_{ 2111 }=&&( s_{5} + s_{31} ) \otimes 1\\
    &&&\qquad
	+ ( s_{3} + s_{4} + s_{21} ) \otimes e_{1}
	+ ( s_{3} + s_{11} ) \otimes e_{2}\\
    &&&\qquad
	+ s_{2} \otimes e_{11}
	+ s_{2} \otimes e_{3}
	+ s_{1} \otimes e_{21}
	+ s_{1} \otimes e_{4}
	+ 1 \otimes e_{31};
\end{aligned}$\end{tcolorbox}

\begin{tcolorbox}$\begin{aligned}
& \F_{ 42 }=&&( s_{6} + s_{41} + s_{22} ) \otimes 1\\
    &&&\qquad
	+ ( s_{5} + s_{4} + s_{3} + s_{2} + s_{31} + s_{21} ) \otimes e_{1}\\
    &&&\qquad
	+ ( s_{2} + s_{4} + s_{21} ) \otimes e_{2}
	+ ( 1+ s_{3} + s_{2} + 2 s_{1} + s_{11} ) \otimes e_{11};
\end{aligned}$\end{tcolorbox}

\begin{tcolorbox}$\begin{aligned}
&\F_{ 321 }=&&( s_{6} + s_{31} + s_{41} + s_{111} ) \otimes 1\\
    &&&\qquad
	+ ( s_{3} + s_{4} + s_{5} + s_{11} + s_{21} + s_{31} ) \otimes e_{1}\\
    &&&\qquad
	+ ( s_{2} + s_{4} + s_{11} + s_{21} ) \otimes e_{2}
	+ ( s_{1} + s_{2} + s_{3} ) \otimes e_{11}\\
    &&&\qquad
	+ ( s_{3} + s_{11} ) \otimes e_{3}
	+ ( 2 s_{1} + s_{2} ) \otimes e_{21}
	+ 1 \otimes e_{111};
\end{aligned}$\end{tcolorbox}

\begin{tcolorbox}$\begin{aligned}
&\F_{ 2211 }=&&( s_{6} + s_{41} + s_{22} ) \otimes 1\\
    &&&\qquad
	+ ( s_{5} + s_{4} + s_{31} + s_{21} ) \otimes e_{1}\\
    &&&\qquad
	+ ( s_{4} + s_{2} \rouge{- s_{11}} + s_{21} ) \otimes e_{2}
	+ ( s_{3} + s_{11} ) \otimes e_{11}\\
    &&&\qquad
	+ ( s_{3} + s_{11} ) \otimes e_{3}
	+ ( s_{2} + s_{1} ) \otimes e_{21}
	+ 1 \otimes e_{22}
	+ s_{2} \otimes e_{4}
	+ s_{1} \otimes e_{31}.
\end{aligned}$\end{tcolorbox}

\noindent Observe  that for $\F_{2211}$, a negative term occurs in the coefficient for $e_2$.
 \begin{property} When $\length(\tau')>\length(\tau)$, for all $k\geq 0$ the coefficient of $e_1^k$ in $\F_\tau$ may be calculated from $\A_\tau$ as:
    \begin{equation}
                \langle \F_\tau,f_{1^k}\rangle = e_k^\perp \A_\tau.
    \end{equation}
\end{property}
\noindent For instance, we have
\begin{tcolorbox}[ams align*]
 & \A_{4321}=s_{10.} + s_{42} + s_{43} + s_{61}  + s_{71} + s_{81} + s_{62}+ s_{311} + s_{411} + s_{511} + s_{1111},\\
  &e_1^\perp \A_{4321}= s_{6} + s_{7} + s_{8} + s_{9} + s_{31} + 2 s_{41}  + 2 s_{51}  + 2 s_{61} + s_{71} \\
			&\hskip5cm  + s_{32} + s_{42} + s_{52} + s_{33}+ s_{111} + s_{211} + s_{311} + s_{411},\\  
  &e_2^\perp \A_{4321}= s_{3} + s_{4} + 2 s_{5} + s_{6} + s_{7} + s_{11} + s_{21} + 2 s_{31}  + s_{41} + s_{51}+ s_{32}, \\
   &e_3^\perp \A_{4321}=   s_{1} + s_{2} + s_{3} + s_{4}, \\
   &e_4^\perp \A_{4321}=  1;
\end{tcolorbox}
and then
\begin{small}
\begin{tcolorbox}[ams align*]
\F_{ 4321 }&= \A_{4321} \otimes 1 
					+ (e_1^\perp \A_{4321}) \otimes e_{1}
					+ ( e_2^\perp \A_{4321} ) \otimes e_{11}
					+ ( e_3^\perp \A_{4321} ) \otimes e_{111}
					+(e_4^\perp \A_{4321})  \otimes e_{1111}\\
&\qquad+ ( s_{4} + s_{6} + s_{8} + s_{21} + s_{31} + 2 s_{41} + s_{51} + s_{61}+ s_{22} + s_{42}  + s_{111} + s_{211} + s_{311} ) \otimes e_{2}\\
&\qquad+ ( s_{4} + s_{7} + s_{21} + s_{31} + s_{41} + s_{51}+ s_{32}  + s_{111} + s_{211} ) \otimes e_{3}\\
&\qquad+ ( 2 s_{2} + 2 s_{3} + s_{4} + 2 s_{5} + s_{6} + 2 s_{11} + 3 s_{21}  + 2 s_{31} + s_{41}+ s_{22} ) \otimes e_{21}\\
&\qquad+ ( s_{6} + s_{31} + s_{41} + s_{111} ) \otimes e_{4}\\
&\qquad+ ( 2 s_{3} + s_{4} + s_{5} + 2 s_{11} + s_{21} + s_{31} ) \otimes e_{31}\\
&\qquad+ ( s_{2} + s_{4} + s_{11} + s_{21} ) \otimes e_{22}\\
&\qquad+ ( 3 s_{1} + 2 s_{2} + s_{3} ) \otimes e_{211}
\end{tcolorbox}
\end{small}
\begin{properties}
For all triangular partitions $\tau$, we have
\begin{enumerate}
\itemsep=5pt
  \item $\A_\tau(\boldsymbol{q}+1)=\sum_{\nu} \langle \F_\tau, f_\nu\rangle$.
  \item For all $k\geq 0$, we have
      \begin{equation}
		\sum_{\length(\nu)=k} \langle \F_{\rouge{\tau}}, f_\nu\rangle=\sum_{\length(\nu)=k} \langle \F_{\rouge{\tau'}}, f_\nu\rangle.
       \end{equation}
   \item For all $n> \length(\tau)$,
   		\begin{equation}
		(\uparrow^{(n)}\!\F_{\tau})(q;\boldsymbol{x})=
			\sum_{\alpha\subseteq \tau} q^{\area(\alpha)}\, s_{(\alpha+1^n)/\alpha}(\boldsymbol{x}).
		\end{equation}
  \item For correctly defined LLT-polynomials $\mathbb{L}_\alpha(t;\boldsymbol{x})$, and $n> \length(\tau)$, we have
     		\begin{equation}
		(\uparrow^{(n)}\!\F_{\tau})(q;\boldsymbol{x})=
			\sum_{\alpha\subseteq \tau} q^{\area(\alpha)}\, \mathbb{L}_\alpha(t;\boldsymbol{x}).
		\end{equation}
 \item For any triangular partitions $\tau_1\subseteq \tau_2$, with $|\tau_2|-|\tau|=1$, then
             		\begin{equation}
		(e_1^\perp \F_{\tau_2})\geq  \F_{\tau_1} \qquad \hbox{(the difference is $e$-positive)}
		\end{equation}
\end{enumerate} 
\end{properties}
We surmise that, for all triangular partition $\tau$, there exists a Tamari-like poset $\mathcal{T}_\tau$ on the sets of partitions $\alpha$ contained in $\tau$, affording a decorated chains enumerating polynomial $\mathcal{T}_\tau(k;\boldsymbol{x})$ in the variable $k$. Here $k$ stands for the length of the chains, with chains ending at $\alpha$ decorated (weighted) by the elementary symmetric function $s_{(\alpha+1^n)/\alpha}$ (or even the LLT-polynomial $\mathbb{L}_\alpha(t;\boldsymbol{x})$). In formula:
\begin{equation}
    \mathcal{T}_\tau(k;\boldsymbol{x}) = \sum_{\length(\gamma)=k} s_{({\rm top}(\gamma)+1^n)/{\rm top}(\gamma)},
\end{equation}
where the sum runs over the set of chains $\gamma$ having length $k$. This $\mathcal{T}_\tau$ is expected to be such that
\begin{property}
    For all triangular partition $\tau$, the difference $\mathcal{T}_\tau(k;\boldsymbol{x})-(\uparrow^{(n)}\!\F_{\tau})(k;\boldsymbol{x})$ is $e$-positive. 
\end{property}
Here, evaluation at $k$ of $(\uparrow^{(n)}\!\F_{\tau})$ corresponds to setting the first $k$ variables $q_i$ equal to $1$, and all remaining ones equal to $0$. For example, we have
\begin{small}
\begin{tcolorbox}[ams align*]
   (\uparrow^{(n)}\!\F_{321})(k;\boldsymbol{x}) &= ({1}/{720})  (k^{6} + 39 k^{5} + 295 k^{4} + 645 k^{3} - 296 k^2 + 36k)\, e_n(\boldsymbol{x})\\ 
  &\qquad  + ({1}/{120}) (k^2 + 3k)  (k^{3} + 27 k^{2} + 74 k - 12)\, e_{(n-1,1)}(\boldsymbol{x}) \\ 
  &\qquad  + ({1}/{24})  (k^{4} + 14 k^{3} + 35 k^2 - 2k) \,  e_{(n-2,2)}(\boldsymbol{x})\\ 
  &\qquad   + ({1}/{6}) (k^{3} + 6 k^2 - k)\, e_{(n-3,3)}(\boldsymbol{x}) \\ 
  &\qquad  + ({1}/{6})(k^{3} + 6 k^2 + 11k)\,   e_{(n-2,11)}(\boldsymbol{x}) \\ 
  &\qquad  +  ({1}/{2}) (k^2 + 5k)\, e_{21} + e_{(n-3,111)}(\boldsymbol{x}).
\end{tcolorbox}
\end{small}


\section{\bleu{Whish list}} 
We would like to have an expansion of $\F_\tau$ of the form
\begin{equation}
  \F_\tau=\sum_{\alpha\subseteq \tau} \mathcal{C}_\alpha \otimes s_{(\alpha+1^{\length(\alpha)})/\alpha},
\end{equation}
with coefficients $\mathcal{C}_\alpha = \mathcal{C}^\tau_\alpha$,
such that
\begin{enumerate}
\item $\mathcal{C}_0 =\A_\tau$,
\item If $\langle \mathcal{C}_\alpha ,s_{\mu}\rangle \neq 0$, then $\sum_{i}i\mu_i \leq  |\tau|-|\alpha|$.
\item For all $\alpha\subseteq \tau$, we have $\langle \mathcal{C}_\alpha ,s_{n}\rangle =1$, where  $n=|\tau|-|\alpha|$,
\item and in some instances $\mathcal{C}^\tau_\alpha=\mathcal{C}^{\tau'}_{\alpha'}$.
\end{enumerate}
 To simplify the presentation of further expressions, let consider 
\begin{align*}
	\F_\tau^{(1)}&:=\A_\tau\otimes 1+ \sum_{0\varsubsetneq \alpha\subseteq \tau} \mathcal{C}^{(1)}_\alpha \otimes s_{(\alpha+1^{\length(\alpha)})/\alpha},
					\qquad {\rm where} \\ 
	\mathcal{C}^{(1)}_\alpha&:= s_n,\quad {\rm with}\ n=|\tau|-|\alpha|. 
\end{align*}
In other words, this is the contribution that corresponds to rules (1) and (2). It does not seem to be possible for the third condition to hold in general. However, it does in the following cases.
\begin{itemize}
\item[$\bullet$] For $\tau=1^k$ and $k$, we have $\F_\tau=\F_\tau^{(1)}$.
\item[$\bullet$] For $\tau=21^k$, we have 
	\begin{equation}
	    \F_\tau=\F_\tau^{(1)}+ \sum_{j=1}^{k-1} s_{(j,1)}\otimes e_{\rho(1^{k-j})}.
	\end{equation}		
\item[$\bullet$] Equivalently, assuming condition (3),  for $\tau =(k+1,1)$ we have 			
	\begin{equation}
	    \F_\tau=\F_\tau^{(1)}+ \sum_{j=1}^{k-1} s_{(j,1)}\otimes e_{\rho({k-j})}.
	\end{equation}			

\item[$\bullet$] For $\tau=321$, we have 
	\begin{align}
	    \F_\tau&=\F_\tau^{(1)}+ s_{31}\otimes e_{\rho({1})} +s_{21}\otimes (e_{\rho({11})}+ e_{\rho({2})})\nonumber \\
	    &\qquad\quad +s_{11}\otimes (e_{\rho({111})} +  e_{\rho({3})} +  e_{\rho({22})}).
	\end{align}			
\item[$\bullet$] For $\tau=421$, we have 
	\begin{align}
	    \F_\tau&=\F_\tau^{(1)}+ (s_{22}+s_{31}+s_{41}+s_{111})\otimes e_{\rho({1})} 
	    			+s_{31}\otimes (e_{\rho({11})}+ e_{\rho({2})})
	    \nonumber \\
	    	    &\qquad\quad +s_{21}\otimes (e_{\rho({111})} + e_{\rho({3})} 
		     			  +  e_{\rho({21})}+  e_{\rho({22})})\nonumber  \\
		    &\qquad\quad + s_{11} \otimes (e_{\rho({211})} +  e_{\rho({31})} +  e_{\rho({4})}).
	\end{align}			
 \end{itemize}

 \nocite{*}

\bibliographystyle{amsplain-ac}
\bibliography{TriangularDiagonalHarmonicsConjecture}

\end{document}